%% file: arxiv.tex
\documentclass[final,3p,times]{elsarticle}
\usepackage[T1]{fontenc}
\usepackage[latin1]{inputenc}
\usepackage[english]{babel}
\usepackage{graphicx,subfigure}
\usepackage{amsmath}
\usepackage{amsthm}
\usepackage{theoremref}
\usepackage{enumerate}
\usepackage{amssymb,amsfonts}
\usepackage{rotfloat}
\theoremstyle{plain}
\newtheorem{theo}{Theorem}
\newtheorem{lem}[theo]{Lemma}

\newtheorem{defi}{Definition}

\newtheorem{remark}{Remark}
\newtheorem{example}{Example}
\newfont{\bbold}{bbold10 scaled\magstep1}
\newcommand{\D}{\displaystyle}

\newcommand{\real}{\mathbb{R}}
\newcommand{\mlt}{\circ}
\newcommand{\mltl}{\star}

\providecommand{\E}{\operatorname{E}}
\providecommand{\N}{\mathbb{N}}

\newcommand{\bO}{\ensuremath{\mathcal{O}}}

\usepackage{pstricks,pst-node,pst-tree,pstricks-add}
\SpecialCoor
\newcommand{\dtree}[1]{\pstree{#1}}
\psset{levelsep=-12pt,nodesep=0pt,treesep=8pt,radius=2pt,tnsep=1pt}

\newcommand{\tnr}[2][dummy]{\TC*[name=#1]~[tnpos=r]{\ensuremath{\scriptstyle #2}}}
\newcommand{\tnl}[2][dummy]{\TC*[name=#1]~[tnpos=l]{\ensuremath{\scriptstyle #2}}}

\begin{document}

\begin{frontmatter}


\author[Darmstadt]{Kristian Debrabant}
\ead{debrabant@mathematik.tu-darmstadt.de}
\author[Trondheim]{Anne Kv{\ae}rn{\o}}
\ead{Anne.Kvarno@math.ntnu.no}
\address[Darmstadt]{Technische Universit\"{a}t Darmstadt, Fachbereich Mathematik, Dolivostra{\ss}e 15,
64293 Darmstadt, Germany}
\address[Trondheim]{Department of Mathematical Sciences, Norwegian
University of Science and Technology, 7491 Trondheim, Norway}
\title{Composition of stochastic B--series with applications to implicit Taylor methods}

\begin{abstract}
In this article, we construct a representation formula for stochastic B--series evaluated in a B--series. This formula is used to give for the first time the order conditions of implicit Taylor methods in terms of rooted trees. Finally, as an example we apply these order conditions to derive in a simple manner a family of strong order 1.5 Taylor methods applicable to It\^{o} SDEs.
\end{abstract}


\end{frontmatter}

\input{intro}
\input{prelim}
\input{taylor}
\input{numerics}

\def\cprime{$'$} \providecommand{\de}[2]{#2}

\end{document}

%% file: intro.tex
\section{Introduction}
Taylor methods have for a long time been a common choice for solving stochastic differential equations (SDEs), and with increased use of automatic differentiation techniques, their popularity will hardly subside. The Taylor expansions of the exact solutions of SDEs, from which the Taylor methods are derived, can take one out of two forms: either as
Wagner--Platen series \cite{kloeden91sai,kloeden99nso} or as B--series
\cite{debrabant08bao,burrage96hso, burrage98goc, komori97rta, roessler06rta}, the relation between the two series has been demonstrated in \cite{debrabant10ste}. In the present paper, we focus on B--series.  In short, the exact solution $X(t)$ evaluated after one step starting at $(t_0,x_0)$ can be expressed in terms of a B--series:
\[ X(t_0+h) = B(\varphi,x_0;h) = \sum_{\tau\in T} \alpha(\tau) \cdot \varphi(\tau)(h) \cdot F(\tau)(x_0), \]
where $\alpha$ is a combinatorial term, $\varphi$ is a stochastic integral, and $F$ is some differential. These series will be thoroughly explained in Subsection \ref{subsec:bseries}. Implicit Taylor methods have been derived e.\,g.\ in \cite{kloeden99nso, tian01itm}. Written in terms of B--series, they take the form
\begin{equation}
\label{eq:STM}
Y_1 = B(\Phi_{ex},x_0;h)+B(\Phi_{im},Y_1;h).
\end{equation}
To illustrate this, consider the semi-implicit Milstein scheme applied to the It\^{o} SDE
\[ dX = f(X)dt + g(X)dW \]
which is given by
\[ Y_1 = x_0+hf(Y_1)+I_{(1)}g(x_0) + I_{(1,1)}(g'g)(x_0). \]
The Milstein method takes the form \eqref{eq:STM} with
\begin{align*}
B(\Phi_{ex},x_0;h) &= x_0+I_{(1)}g(x_0) + I_{(1,1)}(g'g)(x_0) \\
\intertext{ and } B(\Phi_{im},Y_1;h) &= hf(Y_1).
\end{align*}
To compare with the exact solution, the numerical solution $Y_1$ of \eqref{eq:STM} has to be expressed in terms of B--series as well.
Here, this is done in two steps. In Subsection \ref{subsec:comp} we prove a quite general result, well known from deterministic theory: A B--series evaluated in a B--series is again a B--series, written as
\begin{equation}
\label{eq:comp1}
B(\phi_y,B(\phi_x,x_0;h);h) = B(\phi_x\mlt\phi_y,x_0;h).
\end{equation}
This result is not only useful for implicit Taylor methods, but also for composition- and splitting methods. The rule can be used with both  Stratonovich and It\^{o} calculus. In Section \ref{sec:taylor}, we present the B--series for the implicit Taylor methods, and thereby their order conditions. Finally, in Section \ref{seq:num}, a family of implicit Taylor methods  is developed and the results are confirmed by numerical experiments. 

%% file: prelim.tex
\section{Some notation, definitions and preliminary results}\label{sec:prelim}
In this section we introduce some necessary notation and provide a
few definitions and preliminary results that will be used later.
Let $(\Omega,\mathcal{A},\mathcal{P})$ be a probability space. We
denote by $(X(t))_{t \in I}$ the stochastic process which is the
solution of a $d$-dimensional SDE defined by
\begin{equation}\label{eq:SDEdiff}
dX(t)=g_0(X(t))dt+\sum_{l=1}^{m}g_{l}(X(t))\mltl dW_l(t),\quad X(t_0)=x_{0},
\end{equation}
with an $m$-dimensional Wiener process $(W(t))_{t \geq 0}$ and
$I=[t_0,T]$. As usual, \eqref{eq:SDEdiff} is construed as abbreviation of
\begin{equation}\label{SDE}
X(t)=x_{0}+\int_{t_{0}}^{t}g_0(X(s))ds+\sum_{l=1}^{m}\int_{t_{0}}^{t}g_{l}(X(s))\mltl
dW_l(s).
\end{equation}
The integral w.\,r.\,t.\ the Wiener process has to be
interpreted e.\,g.\ as It\^{o} integral with $\mltl dW_l(s)=dW_l(s)$ or
as Stratonovich integral with $\mltl dW_l(s)=\circ dW_l(s)$. We assume that the Borel-measurable coefficients $g_l :
\mathbb{R}^d \rightarrow \mathbb{R}^{d}$ are sufficiently differentiable and that the SDE \eqref{SDE} has a unique solution.

To simplify the presentation, we define $W_0(s)=s$, so that
\eqref{SDE} can be written as
\begin{equation}\label{SDE_new}
X(t)=x_{0}+\sum_{l=0}^{m}\int_{t_{0}}^{t}g_{l}(X(s))\mltl
dW_l(s).
\end{equation}
In the following we denote by $\Xi$ a set of families of measurable mappings,
\[
\Xi:=\big\{\{\varphi(h)\}_{h\geq0}:\;
    \varphi(h):~\Omega\to\real\text{ is }\mathcal{A}-\mathcal{B}\text{-measurable }\forall h\geq0
    \big\}.
\]

\subsection{Stochastic B--series} \label{subsec:bseries}
B--series for deterministic ODEs were introduced by Butcher \cite{butcher63cft}. Today such series appear as a fundamental tool to
do local error analysis on a wide range of problems. B--series for SDEs and their numerical solution by stochastic Runge-Kutta methods have been developed by Burrage and Burrage
\cite{burrage96hso,burrage00oco} to study strong convergence in the
Stratonovich case, by Komori, Mitsui and Sugiura \cite{komori97rta} and Komori \cite{komori07mrt} to study weak convergence in the Stratonovich case and by R\"{o}{\ss}ler
\cite{roessler04ste,roessler06rta} to study weak convergence in both the It\^{o} and the Stratonovich case. However, the distinction between the It\^{o} and the Stratonovich integrals
only depends on the definition of the integrals, not on how the
B--series are constructed. Similarly, the distinction between weak and strong convergence only depends on the definition of the local error. Thus, we presented in \cite{debrabant08bao} a uniform and
self-contained theory for the construction of stochastic B--series for the exact solution of SDEs and its numerical approximation by stochastic Runge-Kutta methods. In this subsection, we review results given there for the solution of \eqref{SDE_new}. In Subsection \ref{subsec:comp} we prove the composition rule \eqref{eq:comp1}.

Following the idea of Burrage and Burrage we introduce the set of
colored, rooted trees related to the SDE \eqref{SDE}, as well as the
elementary differentials associated with each of these trees.

\begin{defi}[Trees] \thlabel{def:trees}
    The set of $m+1$-colored, rooted trees \[T=\{\emptyset\}\cup T_0 \cup T_1 \cup
    \dots \cup T_m\] is recursively defined as follows:
    \begin{description}
    \item[a)] The graph $\bullet_l=[\emptyset]_l$ with only one vertex
        of color $l$ belongs to $T_l$.
    \end{description}
    Let $\tau=[\tau_1,\tau_2,\dotsc,\tau_{\kappa}]_l$ be the tree
    formed by joining the subtrees
    $\tau_1,\tau_2,\dotsc,\tau_{\kappa}$ each by a single branch to a
    common root of color $l$.
    \begin{description}
    \item[b)] If $\tau_1,\tau_2,\dotsc,\tau_{\kappa} \in T$ then
        $\tau=[\tau_1,\tau_2,\dotsc,\tau_{\kappa}]_l \in T_l$.
    \end{description}
\end{defi}
Thus, $T_l$ is the set of trees with an $l$-colored root, and $T$ is
the union of these sets. As usual, the trees are not ordered, i.\,e.\ the order of subtrees is not significant.

\begin{defi}[Elementary differentials] \thlabel{def:el_diff}
    For a tree $\tau \in T$ the elementary differential is
    a mapping $F(\tau):\real^d \rightarrow \real^d$ defined
    recursively by
    \begin{description}
    \item[a)] $F(\emptyset)(x_0)=x_0$, \\ \mbox{}
    \item[b)] $F(\bullet_l)(x_0)=g_l(x_0)$, \\ \mbox{}
    \item[c)] If $\tau=[\tau_1,\tau_2,\dotsc,\tau_{\kappa}]_l \in T_l$
        then
        \[
        F(\tau)(x_0)=g_l^{(\kappa)}(x_0)
        \big(F(\tau_1)(x_0),F(\tau_2)(x_0),\dotsc,F(\tau_{\kappa})(x_0)\big). \]
    \end{description}
\end{defi}

As will be shown in the following, both the exact and the numerical solutions  can formally be written in terms of
B--series.
\begin{defi}[B--series] \thlabel{def:B--series}
 A (stochastic) B--series is
    a formal series of the form
    \[ B(\phi,x_0; h) = \sum_{\tau \in T}
    \alpha(\tau)\cdot\phi(\tau)(h)\cdot F(\tau)(x_0),\] where
    $\phi:T \rightarrow \Xi$ and
    $\alpha: T\rightarrow \mathbb{Q}$ is given by
    \begin{align*}
        \alpha(\emptyset)&=1,&\alpha(\bullet_l)&=1,
        &\alpha(\tau=[\tau_1,\dotsc,\tau_{\kappa}]_l)&=\frac{1}{r_1!r_2!\cdots
            r_{q}! } \prod_{j=1}^{\kappa} \alpha(\tau_j),
    \end{align*}
    where $r_1,r_2,\dots,r_{q}$ count equal trees among
    $\tau_1,\tau_2,\dots,\tau_{\kappa}$.
\end{defi}

Note that $\alpha(\tau)$ is the inverse of the order of the automorphism group of $\tau$. To simplify the presentation, in the following we assume that all elementary
differentials exist and all considered B--series converge. Otherwise, one has
to consider truncated B--series and discuss the remainder term.

The next lemma proves that if $Y(h)$ can be written as a B--series,
then $f(Y(h))$ can be written as a similar series, where the sum is
taken over trees with a root of color $f$ and subtrees in $T$.  The
lemma is fundamental for deriving B--series for the exact and the
numerical solution. It can also be used for deriving weak convergence
results.
\begin{lem} \thlabel{lem:f_y} If $Y(h)=B(\phi, x_0; h)$ with $\phi(\emptyset) \equiv 1$ is some
    B--series and $f\in C^{\infty}(\real^d,\real^{\hat{d}})$ then
    $f(Y(h))$ can be written as a formal series of the form
    \begin{equation}
        \label{eq:f}
        f(Y(h))=\sum_{u\in U_f} \beta(u)\cdot \psi_\phi(u)(h)\cdot G(u)(x_0),
    \end{equation}
    where $U_f$ is a set of trees derived from $T$ by
    \begin{description}
    \item[a)] $[\emptyset]_f \in U_f$, and if
        $\tau_1,\tau_2,\dotsc,\tau_{\kappa}\in T$ then
        $[\tau_1,\tau_2,\dotsc,\tau_{\kappa}]_f\in U_f$, \\ \mbox{}
    \item[b)] $G([\emptyset]_f)(x_0)=f(x_0)$ and $G(u=[\tau_1,\dotsc,\tau_{\kappa}]_f)(x_0) =
        f^{(\kappa)}(x_0)
        \big(F(\tau_1)(x_0),\dotsc,F(\tau_{\kappa})(x_0)\big)$, \\
        \mbox{}
    \item[c)] $\beta([\emptyset]_f)=1$ and $\D
        \beta(u=[\tau_1,\dotsc,\tau_{\kappa}]_f)
        =\frac{1}{r_1!r_2!\cdots r_{q}!}\prod_{j=1}^{\kappa}
        \alpha(\tau_{{j}})$, \\where $r_1,r_2,\dots,r_{q}$ count
        equal trees among $\tau_1,\tau_2,\dotsc,\tau_{\kappa}$, \\
        \mbox{}
    \item[d)] $\psi_\phi([\emptyset]_f)\equiv1$ and
        $\psi_\phi(u=[\tau_1,\dotsc,\tau_{\kappa}]_f)(h) =
        \prod_{j=1}^{\kappa} \phi(\tau_j)(h)$.
    \end{description}
\end{lem}
\begin{proof}
    Writing $Y(h)$ as a B--series, we have
    \begin{align*}
        f(Y(h)) &= f\left(\sum_{\tau\in T}\alpha(\tau) \cdot
          \phi(\tau)(h)
          \cdot F(\tau)(x_{0})\right) \\
        &= \sum_{\kappa=0}^{\infty}
        \frac{1}{\kappa!}f^{(\kappa)}(x_{0}) \left(\sum_{\tau\in
              T\setminus\{\emptyset\}}\alpha(\tau) \cdot \phi(\tau)(h)
          \cdot F(\tau)(x_{0}) \right)^{\kappa}  \\
        & = f(x_{0}) + \sum_{\kappa=1}^{\infty}\frac{1}{\kappa!}
        \sum_{\{\tau_{1},\tau_{2},\dotsc \tau_{\kappa}\}\in
            T\backslash \{\emptyset\}}
        \frac{\kappa!}{r_{1}!r_{2}!\cdots r_{q}!}
         \cdot  \left(\prod_{j=1}^{\kappa}
          \alpha(\tau_{j}) \cdot \phi(\tau_j)(h)\right)
        f^{(\kappa)}(x_{0})\big(F(\tau_{1})(x_{0}),\dotsc,
        F(\tau_{\kappa})(x_{0}) \big)
    \end{align*}
    where the last sum is taken over all possible unordered combinations of
    $\kappa$ trees in $T$. For each set of trees
    $\tau_{1},\tau_{2},\dotsc,\tau_{\kappa} \in T$ we assign a
    $u=[\tau_{1},\tau_{2},\dotsc,\tau_{\kappa}]_{f} \in U_{f}$. The
    theorem is now proved by comparing term by term with \eqref{eq:f}.
\end{proof}
When Lemma \ref{lem:f_y} is applied to the functions $g_{l}$ on the
right hand side of \eqref{SDE_new} we get the following result: If
$Y(h)=B(\phi, x_0; h)$ then
\begin{equation}
    \label{eq:g-series}
    g_l(Y(h)) = \sum_{\tau\in T_l}\alpha(\tau) \cdot \phi'_l(\tau)(h)\cdot
    F(\tau)(x_0)
\end{equation}
in which
\[ \phi'_l(\tau)(h) =
\begin{cases}
    1 & \text{if } \tau=\bullet_{l}, \\
    \D \prod_{j=1}^{\kappa} \phi(\tau_j)(h) & \text{if} \;
    \tau=[\tau_1,\dotsc,\tau_{\kappa}]_l \in T_l.
\end{cases} \]

\begin{theo} \thlabel{thm:Bex} The solution $X(t_0+h)$ of \eqref{SDE_new}
    can be written as a B--series $B(\varphi,x_0; h)$ with
    \begin{align*}
        \varphi(\emptyset)&\equiv1,&\varphi(\bullet_l)(h)&=W_l(h), &
        \varphi(\tau=[\tau_1,\dotsc,\tau_{\kappa}]_l)(h)&=\int_0^h
        \prod_{j=1}^{\kappa} \varphi(\tau_j)(s)\mltl dW_l(s).
    \end{align*}
\end{theo}
\begin{proof}
Write the exact solution as some B--series
$X(t_0+h)=B(\varphi,x_0;h)$ with $\varphi(\emptyset)\equiv1$. By \eqref{eq:g-series} the SDE \eqref{SDE_new} can be written as
\[ \sum_{\tau\in T} \alpha(\tau) \cdot \varphi(\tau)(h) \cdot
F(\tau)(x_0) = x_0 + \sum_{l=0}^m \sum_{\tau \in T_l} \alpha(\tau) \cdot \int_0^h
\varphi_l'(\tau)(s) \mltl dW_l(s) \cdot F(\tau)(x_0). \]
Comparing term by term we get
\[ \varphi(\tau)(h) =
\int_0^h \varphi'(\tau)(s)\mltl dW_l(s) \quad \text{for} \quad\tau\in
T_l, \quad l=0,1,\dotsc,m. \]
The proof is completed by induction on $\tau$.
\end{proof}

The same result is given for the Stratonovich case in
\cite{burrage00oco,komori07mrt}, but it clearly also applies to the
It\^{o} case.

The definition of the
order of the tree, $\rho(\tau)$, is motivated by the fact that $\E
W_{l}(h)^{2}=h$ for $l\geq 1 $ and $W_{0}(h)=h$:
\begin{defi}[Order]
    The order of a tree $\tau \in T$ respectively  $u\in U_f$ is defined by
    \[\rho(\emptyset)=0,\quad\rho(u=[\tau_1,\dots,\tau_\kappa]_f)=\sum\limits_{i=1}^\kappa\rho(\tau_i),\]
    and
    \[\rho(\tau=[\tau_1,\dots,\tau_\kappa]_l)=\sum\limits_{i=1}^\kappa\rho(\tau_i)+
    \begin{cases}
        1&\text{for }l=0,\\
        \frac12&\text{otherwise}.
    \end{cases}
    \]
\end{defi}
We conclude this subsection with an example.
\begin{example}
For $\tau  =
\raisebox{-3ex}{\dtree{\tnr{0}}{\tnl{1} \dtree{\tnr{1}}{\tnl{2} \tnr{2}}}}
$ we get
\[ \rho(\tau)= 3,\qquad \alpha(\tau)= \frac{1}{2}, \qquad F(\tau)(x_0)= g_0''(g_1,g_1''(g_2,g_2))(x_0) \] and
\begin{align*}
  \varphi(\tau)(h) & =
  \int_{0}^h W_1(s_1)\left(\int_{0}^{s_1}W_2(s_2)^2\mltl
    dW_1(s_2)\right) ds_1  \\ & =
 \begin{cases}
    4J_{(2,2,1,1,0)}+2J_{(2,1,2,1,0)}+2J_{(1,2,2,1,0)}
     & (\text{Stratonovich}), \\[2mm]
    4I_{(2,2,1,1,0)}+2I_{(2,1,2,1,0)}+2I_{(1,2,2,1,0)} & \\ +2I_{(0,1,1,0)}+2I_{(2,2,0,0)}+I_{(1,0,1,0)}+I_{(0,0,0)}
    & (\text{It\^{o}}).
 \end{cases}
\end{align*}

\end{example}

\subsection{Composition of B--series} \label{subsec:comp}
In this section, we will show that, similar to the deterministic case, B--series evaluated in a B--series can be written as a B--series, or more specific:
\begin{equation} \label{eq:comp}
   B(\phi_y,B(\phi_x,x_0;h);h) = B(\phi_x \mlt \phi_y,x_0; h),
\end{equation}
and we will define the composition operator $\phi_x \mlt \phi_y$.

For this purpose, we will need the decomposition of a tree $\tau\in T$ into a subtree $\vartheta$ sharing the root with $\tau$, and a remainder multiset $\omega$ of trees left over when $\vartheta$ is removed from $\tau$. This forms a triple $(\tau,\vartheta,\omega)$. We also include the empty tree as a possible subtree, in which case the triple becomes $(\tau,\emptyset,\{\tau\})$.
\begin{example} \thlabel{ex:triples}
Two examples of such triples are
\[
\left(\dtree{\tnr{j}}{\tnl{j} \tnr{j} \dtree{\tnr{j}}{\tnl{j} \tnr{j}}},
       \dtree{\tnr{j}}{\tnl{j} \tnr{j}},
       \{\tnr{j},\tnr{j},\tnr{j}\}
\right),\]
corresponding to
\[
\tau=\dtree{\tnr[A]{j}}{\tnl[B]{j} \tnr[C]{j} \dtree{\tnr[D]{j}}{\tnl[E]{j} \tnr[F]{j}}}
\pspolygon[linearc=5pt,
        linestyle=dotted,dotsep=2pt,
        linewidth=0.5pt]%
  (!\psGetNodeCenter{B} B.x 0.4 sub B.y 0.1 add)%
  (!\psGetNodeCenter{B} B.x B.y 0.2 add)%
  (!\psGetNodeCenter{A} A.x A.y 0.1 add)%
  (!\psGetNodeCenter{D} D.x 0.1 sub D.y 0.1 add)%
  (!\psGetNodeCenter{D} D.x 0.6 add D.y 0.1 add)%
  (!\psGetNodeCenter{A} A.x A.y 0.4 sub)%
\rput(!\psGetNodeCenter{A} A.x 0.7 add A.y 0.1 sub){\rnode{G}{\hat{=}\vartheta}}
\pspolygon[linearc=5pt,
        linestyle=dashed,dash=3pt 2pt,
        linewidth=0.5pt]%
  (!\psGetNodeCenter{C} C.x 0.2 sub C.y 0.2 sub)%
  (!\psGetNodeCenter{E} E.x 0.2 sub E.y 0.2 add)%
  (!\psGetNodeCenter{F} F.x 0.4 add F.y 0.2 add)%
  (!\psGetNodeCenter{D}\psGetNodeCenter{F} F.x 0.2 add D.y 0.2 add)%
  (!\psGetNodeCenter{D} D.x 0.2 sub D.y 0.2 add)%
  (!\psGetNodeCenter{C} C.x 0.2 add C.y 0.2 sub)
  \rput(!\psGetNodeCenter{F} F.x 0.7 add F.y 0.1 sub){\rnode{H}{\hat{=}\omega}}
  \qquad\text{or equivalently}\qquad
\tau=\dtree{\tnr[A]{j}}{\tnl[B]{j} \tnr[C]{j} \dtree{\tnr[D]{j}}{\tnl[E]{j} \tnr[F]{j}}}
\pspolygon[linearc=5pt,
        linestyle=dotted,dotsep=2pt,
        linewidth=0.5pt]%
  (!\psGetNodeCenter{C} C.x 0.2 sub C.y 0.2 add)%
  (!\psGetNodeCenter{D} D.x 0.1 sub D.y 0.1 add)%
  (!\psGetNodeCenter{D} D.x 0.6 add D.y 0.1 add)%
  (!\psGetNodeCenter{A} A.x A.y 0.4 sub)%
\rput(!\psGetNodeCenter{A} A.x 0.7 add A.y 0.1 sub){\rnode{G}{\hat{=}\vartheta}}
\pspolygon[linearc=5pt,
        linestyle=dashed,dash=3pt 2pt,
        linewidth=0.5pt]%
  (!\psGetNodeCenter{B} B.x 0.2 sub B.y 0.2 add)%
  (!\psGetNodeCenter{E} E.x 0.2 sub E.y 0.2 add)%
  (!\psGetNodeCenter{F} F.x 0.4 add F.y 0.2 add)%
  (!\psGetNodeCenter{D}\psGetNodeCenter{F} F.x 0.2 add D.y 0.2 add)%
  (!\psGetNodeCenter{C} C.x 0.2 sub C.y 0.2 add)%
  (!\psGetNodeCenter{B} B.x 0.1 add B.y 0.2 sub)%
  (!\psGetNodeCenter{B} B.x 0.4 sub B.y 0.2 sub)
  \rput(!\psGetNodeCenter{F} F.x 0.7 add F.y 0.1 sub){\rnode{H}{\hat{=}\omega}},
\]
and
\[
\left(\dtree{\tnr{j}}{\tnl{j} \tnr{j} \dtree{\tnr{j}}{\tnl{j} \tnr{j}}},
       \dtree{\tnr{j}}{\tnl{j} \tnr{j}},
       \left\{\dtree{\tnr{j}}{\tnl{j} \tnr{j}}\right\}
\right),
\]
corresponding to
\[
\tau=\dtree{\tnr[A]{j}}{\tnl[B]{j} \tnr[C]{j} \dtree{\tnr[D]{j}}{\tnl[E]{j} \tnr[F]{j}}}
\pspolygon[linearc=5pt,
        linestyle=dotted,dotsep=2pt,
        linewidth=0.5pt]%
  (!\psGetNodeCenter{B} B.x 0.7 sub B.y 0.2 add)%
  (!\psGetNodeCenter{C} C.x 0.3 add C.y 0.1 add)%
  (!\psGetNodeCenter{A} A.x 0.3 add A.y 0.2 sub)%
  (!\psGetNodeCenter{A} A.x A.y 0.2 sub)%
\rput(!\psGetNodeCenter{A} A.x 0.7 add A.y 0.1 sub){\rnode{G}{\hat{=}\vartheta}}
\pspolygon[linearc=5pt,
        linestyle=dashed,dash=3pt 2pt,
        linewidth=0.5pt]%
  (!\psGetNodeCenter{E} E.x 0.6 sub E.y 0.2 add)%
  (!\psGetNodeCenter{F} F.x 0.4 add F.y 0.1 add)%
  (!\psGetNodeCenter{D} D.x 0.2 add D.y 0.2 sub)%
  (!\psGetNodeCenter{D} D.x D.y 0.2 sub)%
  \rput(!\psGetNodeCenter{F} F.x 0.7 add F.y 0.1 sub){\rnode{H}{\hat{=}\omega}}.
\]
\end{example}
From this example we observe that the same pair of a tree $\tau$ and a subtree $\vartheta$ may have more than one remainder multiset. We then define
$ST(\tau)$ as the set of all possible subtrees of $\tau$ together with the corresponding  remainder multiset $\omega$, that is, for each $\tau\in T\backslash \emptyset$ we have
\begin{align*}
ST(\bullet_l) &= \{ (\emptyset,\{\bullet_l\}), (\bullet_l,\{\emptyset\}) \}, \\
ST(\tau =[\tau_1,\dotsc,\tau_{\kappa}]_l) &=
\{(\vartheta,\omega)\;:\; \vartheta=[\vartheta_1,\dotsc,\vartheta_{\kappa}]_l, \quad
   \omega=\{\omega_1,\dotsc,\omega_{\kappa}\}, \\ &\qquad
   (\vartheta_i,\omega_i) \in ST(\tau_i), \quad i=1,\dots,\kappa \} \cup
   (\emptyset,\{\tau\}).
\end{align*}
Further, for each $\vartheta\in T$ define
\[ A(\vartheta) = \{ (\tau,\omega)\;: \; (\vartheta,\omega)\in ST(\tau)\}. \]
At last, we will use that
\begin{equation} \label{eq:switchsum}
\left\{(\tau,\vartheta,\omega) \;:\; \vartheta\in T, (\tau,\omega)\in A(\vartheta)\right\}
= \left\{(\tau,\vartheta,\omega)\; :\; \tau\in T, (\vartheta,\omega)\in ST(\tau) \right\}.
\end{equation}
And with these preparations, we are ready to prove \eqref{eq:comp}.
From Definition \ref{def:B--series} we have
\begin{equation} \label{eq:bcom_tmp}
B(\phi_y,B(\phi_x,x_0;h);h) = \sum_{\vartheta\in T} \alpha(\vartheta) \cdot
\phi_y(\vartheta)(h) \cdot F(\vartheta)(B(\phi_x,x_0;h)).
\end{equation}
We are now looking for an expression for $F(\vartheta)$ of the form
\begin{equation} \label{eq:finB}
F(\vartheta)(B(\phi_x,x_0;h)) = \sum_{(\tau,\omega)\in A(\vartheta)} \gamma(\tau,\vartheta,\omega) \frac{\alpha(\tau)}{\alpha(\vartheta)}\cdot \bar{\phi}_x(\tau,\omega)(h)\cdot F(\tau)(x_0).
\end{equation}
In this case, we insert \eqref{eq:finB} into \eqref{eq:bcom_tmp} and apply \eqref{eq:switchsum} to obtain \eqref{eq:comp} with
\begin{equation} \label{eq:comp_tmp}
\left(\phi_x\mlt\phi_y \right)(\tau) = \sum_{(\vartheta,\omega)\in ST(\tau)} \gamma(\tau,\vartheta,\omega) \phi_y(\vartheta)\bar{\phi}_x(\tau,\omega).
\end{equation}
To find recursive formulas for the still unknown quantities $\gamma$ and $\bar{\phi}_x$ we start with the trivial case $\vartheta=\emptyset$.
We have $A(\emptyset) = \{(\tau,\{\tau\})\;
:\;\tau\in T\}$ and
\[ F(\emptyset)(B_x)=B_x \]
where, for simplicity, we use the abbreviation $B_x=B(\phi_x,x_0;h)$.
Thus we can choose $\gamma(\tau,\emptyset,\{\tau\})=1$ and $\bar{\phi}_x(\tau,\{\tau\})=\phi_x(\tau)$.
For $\vartheta=\bullet_l$ we have
$A(\bullet_l)=\{(\tau,\omega) : \tau=[\tau_1,\dotsc,\tau_{\kappa}]_l, \;
\omega = \{\tau_1,\dotsc,\tau_{\kappa}\} \; \}$.
By \thref{lem:f_y} and \thref{def:trees,def:el_diff,def:B--series} we get
\[
F(\bullet_l)(B_x) = g_l(B_x) =
\sum_{\tau\in T_l} \alpha(\tau)\cdot \prod_{i=1}^{\kappa} \phi_x(\tau_i)(h) \cdot F(\tau)(x_0),
\]
so we can choose $\gamma(\tau,\vartheta,\omega)=1$ and $\bar{\phi}_x(\tau,\omega)= \prod_{i=1}^{\kappa} \phi_x(\tau_i)$.
For the more general case, let $\vartheta=[\vartheta_1,\dotsc,\vartheta_{\kappa_{\vartheta}}]_l$, so that
\[ F(\vartheta)(B_x) = g_l^{(\kappa_{\vartheta})}(B_x)\left(
   F(\vartheta_1)(B_x),\dotsc,F(\vartheta_{\kappa_{\vartheta}})(B_x)\right).\]
Apply Lemma \ref{lem:f_y} on $g_l^{(\kappa_{\vartheta})}(B_x)$, use the induction hypothesis \eqref{eq:finB} on
each $F(\vartheta_i)(B_x)$ and the multilinearity of $g_l^{(\kappa_{\vartheta})}$ to obtain
\begin{equation}\label{eq:sumtheta}
\begin{split}
F(\vartheta)(B_x) =& \sum_{\omega_0\in U_{g_l^{\kappa_{\vartheta}}}}
\sum_{(\tau_1,\omega_1)\in A(\vartheta_1)} \dotsm
\sum_{(\tau_{\kappa_{\vartheta}},\omega_{\kappa_{\vartheta}})\in A(\vartheta_{\kappa_{\vartheta}})}
\left(\beta(\omega_0) \prod_{i=1}^{\kappa_{\vartheta}} \gamma(\tau_i,\vartheta_i,\omega_i)
\frac{\alpha(\tau_i)}{\alpha(\vartheta_i)} \right)
\cdot \left( \psi_{\phi_x}(\omega_0)\prod_{i=1}^{\kappa_{\vartheta}} \bar{\phi}_x(\tau_i,\omega_i)\right)(h) \\
& \cdot g_{l}^{(\kappa_0+\kappa_{\vartheta})}(x_0)
\left(F(\delta_1)(x_0),\dotsc,F(\delta_{\kappa_0})(x_0),F(\tau_1)(x_0),\dotsc,F(\tau_{\kappa_{\vartheta}})(x_0) \right),
\end{split}
\end{equation}
where each $\omega_0=[\delta_1,\dotsc,\delta_{\kappa_0}]_{g_l^{\kappa_{\vartheta}}}$.
The last term can be recognized as the elementary differential
$F(\tau)(x_0)$ for $\tau=[\delta_1,\dotsc,\delta_{\kappa_0},\tau_1,\dotsc,\tau_{\kappa_{\vartheta}}]_l$.
Choosing
\[ \bar{\phi}_x(\tau,\omega) = \left(\prod_{i=1}^{\kappa_0}\phi_x(\delta_i) \right) \prod_{i=1}^{\kappa_{\vartheta}} \bar{\phi}_x(\tau_i,\omega_i) \]
we obtain by induction
\[  \bar{\phi}_x(\tau,\omega) =
\prod_{\delta\in\omega}\phi_x(\delta). \]
To complete the argument we have to find the expression for $\gamma(\tau,\vartheta,\omega)$ for this $(\tau,\vartheta,\omega)$ triple. For this purpose,
write $\vartheta=[\bar{\vartheta}_1^{r_1},\dotsc,\bar{\vartheta}_{q}^{r_q}]_l$, in the sense that $\vartheta$ is composed by $q$ different trees, each appearing $r_i$ times. Further, let $\tau_{ik}\in\{\tau_1,\dots,\tau_\kappa\}$ be the distinct trees with multiplicity $r_{ik}$, $k=1,\dots,p_i$, of which the $\bar{\vartheta}_i$ are subtrees, $\sum_{k=1}^{p_i}r_{ik}=r_i$.
Finally, let $\bar\delta_i\in\omega$ be the distinct trees with multiplicity $s_i$, $i=1,\dots,p_0$, of the remainder multiset which are directly connected to the root of $\tau$, $\sum_{k=1}^{p_0} s_k=\kappa_{0}$.
Then, $\tau$ can be written as
\begin{equation}
\label{eq:alltrees}
\tau=[\bar\delta_1^{s_1},\dotsc, \bar\delta_{p_0}^{s_{p_0}},
   \tau_{11}^{r_{11}},\dotsc,\tau_{1p_1}^{r_{1p_1}} ,\dotsc,
   \tau_{q1}^{r_{q1}},\dotsc,\tau_{qp_q}^{r_{qp_q}}]_l=
   [ \bar{\tau}_1^{R_1},\dotsc,\bar{\tau}_{Q}^{R_Q}]_l,
\end{equation}
where the rightmost expression above indicates that $\tau$ is composed by $Q$ different trees each appearing $R_i$ times.
Now the terms corresponding to each of the subtrees $[\tau_{i1}^{r_{i1}} ,\dotsc,\tau_{ip_i}^{r_{ip_i}}]_l$ appear exactly $r_i!/(r_{i1}!\dotsm r_{ip_i}!)$ times in the sum \eqref{eq:sumtheta}.
Using $\beta(\omega_0)$ from Lemma \ref{lem:f_y} the combinatorial term in \eqref{eq:comp_tmp} becomes
\[
\gamma(\tau,\vartheta,\omega)\frac{\alpha(\tau)}{\alpha(\vartheta)}
= \left(\frac{1}{s_1!\dotsm s_{p_0}!}\prod_{i=1}^{\kappa_0} \alpha(\delta_i) \right)
\left( \prod_{i=1}^q \frac{r_i!}{r_{i1}!\dotsm r_{ip_i}!} \right)
\left( \prod_{i=1}^{\kappa_{\vartheta}} \gamma(\tau_i,\vartheta_i,\omega_i) \frac{\alpha(\tau_i)}{\alpha(\vartheta_i)} \right).
\]
From the  definition of $\alpha$ for $\tau$ and $\vartheta$ we obtain
\begin{equation}\label{eq:gamma}
\gamma(\tau,\vartheta,\omega)=\frac{R_1!\dotsm R_Q!}{s_1!\dotsm s_{p_0}! r_{11}!\dotsm r_{qp_q}!} \prod_{i=1}^{\kappa_{\vartheta}} \gamma(\tau_i,\vartheta_i,\omega_i).
\end{equation}
Let us illustrate the use of this formula by an example:
\begin{example}
Let $\tau=\dtree{\tnr{j}}{\tnl{j} \tnr{j}}$. Then for
\begin{align}
\vartheta &=\bullet_j, \text{ we have } \omega=\{\bullet_j,\bullet_j\} \text{ and }
\gamma(\tau,\vartheta,\omega)= 1,  \nonumber \\
 \intertext{and for}
\vartheta &= \dtree{\tnr{j}}{\tnr{j}}\,, \text{ we have }  \omega=\{\bullet_j\}
\text{ and } \gamma(\tau,\vartheta,\omega)=2 \label{eq:gamma2bush}
\,.
\end{align}
Next consider the tree
\[ \tau=\dtree{\tnr{j}}{\tnl{j} \tnl{j} \tnr{j} \dtree{\tnr{j}}{\tnr{j} \tnr{j}}}. \]
One possible partition of this tree is
\begin{equation}
\label{eq:expart1}
\vartheta=\dtree{\tnr{j}}{\tnl{j} \tnr{j}}, \quad
\omega = \{\bullet_j, \bullet_j, \bullet_j, \bullet_j\}.
\end{equation}
Then \eqref{eq:alltrees} becomes
\[ \tau = \left[ \bullet_j^2,\bullet_j,\dtree{\tnr{j}}{\tnl{j} \tnr{j}} \right]_j = \left[ \bullet_j^3,\dtree{\tnr{j}}{\tnl{j} \tnr{j}} \right]_j,\]
and we  get $\kappa_0=2$, $p_0=1$, $s_1=2$, $\kappa_{\vartheta}=2$, $r_{11}=r_{12}=1$, $Q=2$ and $R_1=3$, $R_2=1$, so for the partition \eqref{eq:expart1} of $\tau$ we get
\[ \gamma(\tau,\vartheta,\omega) = 3. \]
For the partition
\begin{equation}
\label{eq:expart2}
\vartheta=\dtree{\tnr{j}}{\dtree{\tnr{j}}{\tnr{j}}}\;,\quad
   \omega = \{\bullet_j,\bullet_j,\bullet_j,\bullet_j\},
\end{equation}
\eqref{eq:alltrees} becomes
\[ \tau =\left[ \bullet_j^3,\dtree{\tnr{j}}{\tnl{j} \tnr{j}} \right]_j
= \left[ \bullet_j^3,\dtree{\tnr{j}}{\tnl{j} \tnr{j}} \right]_j, \]
giving $\kappa_0=3$, $p_0=1$ so $s_1=3$. Further $\kappa_{\vartheta}=1$, and $Q$, $R_1$ and $R_2$ are as above. For the subtree $[\bullet_j,\bullet_j]_j$ we have to use the partition \eqref{eq:gamma2bush}, thus for the partition \eqref{eq:expart2} we get
\[ \gamma(\tau,\vartheta,\omega) = 2.\]
\end{example}
We can now state the main result of this section:
\begin{theo}[Composition of B--series] \thlabel{thm:comp}
Let $\phi_x,\phi_y\;:\;T \rightarrow \Xi$ and $\phi_x(\emptyset)\equiv 1$. Then the B--series $B(\phi_y,\cdot;h)$ evaluated at $B(\phi_x,x_0;h)$ is again a B--series,
\[ B(\phi_y,B(\phi_x,x_0;h);h)=B(\phi_x\mlt \phi_y,x_0;h),\]
where
\[ \left(\phi_x\mlt \phi_y \right)(\tau) = \sum_{(\vartheta,\omega)\in ST(\tau)} \gamma(\tau,\vartheta,\omega)\left( \phi_y(\vartheta)\prod_{\delta \in \omega} \phi_x(\delta)\right)
\]
with $\gamma(\emptyset,\emptyset,\{\emptyset\})=1$ and $\gamma(\tau, \vartheta,\omega)$ given by \eqref{eq:gamma} for $\tau\neq\emptyset$.
\end{theo}
\begin{remark}
The composition rule $\phi_x\mlt \phi_y$ can also be expressed in terms of ordered trees, as in the deterministic case, see e.\,g.\ \cite{hairer06gni}.
The combinatorial term $\gamma(\tau,\vartheta,\omega)$ is then the number of equal terms appearing in the sum over ordered subtrees.
\end{remark}
\begin{example} For example
\begin{align*}
\left(\phi_x\mlt \phi_y\right)\left(\dtree{\tnr{j}}{\tnl{l}\tnr{l}\tnr{j}}\right)
&= \phi_y(\emptyset)\phi_x \left(\dtree{\tnr{j}}{\tnl{l}\tnr{l}\tnr{j}}\right) + \phi_y(\bullet_j)\phi_x(\bullet_l)^2\phi_x(\bullet_j) \\&
+ 2\phi_y\left(\dtree{\tnr{j}}{\tnr{l}}\right)\phi_x(\bullet_l)\phi_x(\bullet_j)
+ \phi_y\left(\dtree{\tnr{j}}{\tnr{j}}\right)\phi_x(\bullet_l)^2 \\&
+ 2\phi_y\left(\dtree{\tnr{j}}{\tnl{l}\tnr{j}}\right)\phi_x(\bullet_l)
+ \phi_y\left(\dtree{\tnr{j}}{\tnl{l}\tnr{l}}\right)\phi_x(\bullet_j)
+ \phi_y\left(\dtree{\tnr{j}}{\tnl{l}\tnr{l}\tnr{j}}\right),
\end{align*} and
\[
(\phi_x\mlt\phi_y)\left(\raisebox{-1.5ex}{\dtree{\tnr{j}}{\dtree{\tnr{k}}{\tnr{l}}}}\right)
=
\phi_y(\emptyset)\phi_x \left(\raisebox{-1.5ex}{\dtree{\tnr{j}}{\dtree{\tnr{k}}{\tnr{l}}}}\right)
+ \phi_y(\bullet_j)\phi_x\left(\dtree{\tnr{k}}{\tnr{l}}\right)
+ \phi_y\left(\dtree{\tnr{j}}{\tnr{k}}\right)\phi_x(\bullet_l)
+\phi_x \left(\raisebox{-1.5ex}{\dtree{\tnr{j}}{\dtree{\tnr{k}}{\tnr{l}}}}\right).
\]
\end{example}


%% file: taylor.tex
\section{Application to implicit Taylor methods}
\label{sec:taylor}
We will now see how the composition rule will give the order conditions for implicit Taylor methods. One step of this method is given by
\begin{equation}
\label{eq:one_step_taylor}
Y_1 = B(\Phi_{ex},x_0;h) + B(\Phi_{im},Y_1;h)
\end{equation}
with $\Phi_{ex}(\emptyset)\equiv 1$ and $\Phi_{im}(\emptyset)\equiv 0$. If the solution $Y_1$ is expanded into a B--series, $B(\Phi,x_0;h)$, then $\Phi(\emptyset)=\Phi_{ex}(\emptyset)+\Phi_{im}(\emptyset)\equiv1$ implies that \eqref{eq:one_step_taylor} can be written as
\begin{align}
B(\Phi,x_0;h) &=B(\Phi_{ex},x_0;h)+B(\Phi_{im},B(\Phi,x_0;h);h) \nonumber \\
              &=B(\Phi_{ex},x_0;h)+B(\Phi\mlt\Phi_{im},x_0;h).
              \label{eq:b_taylor}
\end{align}
Comparing term by term gives us the following result:
\begin{theo} \thlabel{thm:Bnum}
The numerical solution $Y_1$ given by \eqref{eq:one_step_taylor} can be written as a B--series
\[ Y_1 = B(\Phi,x_0;h) \]
with $\Phi$ recursively defined by
\begin{subequations}\label{eq:phi_taylor}
\begin{align}
\Phi(\emptyset) & \equiv 1, \\
\Phi(\tau) &= \Phi_{ex}(\tau) + (\Phi\mlt\Phi_{im})(\tau). \label{eq:phi_taylor_b}
\end{align}
\end{subequations}
\end{theo}
Since $\Phi_{im}(\emptyset)\equiv 0$ the expression \eqref{eq:phi_taylor_b} can be written as
\[ \Phi(\tau)=\Phi_{ex}(\tau)+\Phi_{im}(\tau) + \mathcal{R}(\tau) \]
with
\[ \mathcal{R}(\tau)= \sum_{\substack{(\vartheta,\omega)\in ST(\tau)\\\vartheta \not= \emptyset,\tau}} \gamma(\tau,\vartheta,\omega)\left( \Phi_{im}(\vartheta)\prod_{\delta \in \omega} \Phi(\delta)\right).
\]
In Table \ref{tab:ordcond} the $\mathcal{R}$ functions are listed for all trees with up to four nodes.
\begin{sidewaystable}
\caption{The correction terms $\mathcal{R}$ together with the weight functions $\varphi(\tau)$ of the correct solution for all trees with up to four nodes. } \label{tab:ordcond}
\[\begin{array}{ccc}
\tau & \mathcal{R}(\tau) & \varphi(\tau) \\ \hline
 \bullet_i & 0 & W_i \\[2ex]
 \dtree{\tnr{i}}{\tnr{j}} & \Phi_{im}(\bullet_i)\Phi(\bullet_j) & \int W_j\mltl dW_i \\[2ex]
 \dtree{\tnr{i}}{\tnl{j}\tnr{k}} &  \Phi_{im}(\bullet_i)\Phi(\bullet_j)\Phi(\bullet_k)
 + \Phi_{im}\left(\dtree{\tnr{i}}{\tnr{j}}\right)\Phi(\bullet_k)
 + \Phi_{im}\left(\dtree{\tnr{i}}{\tnr{k}}\right)\Phi(\bullet_j) &
 \int W_jW_k\mltl dW_i \\[2ex]
 \raisebox{-1.5ex}{\dtree{\tnr{i}}{\dtree{\tnr{j}}{\tnr{k}}}} &
 \Phi_{im}(\bullet_i)\Phi\left(\dtree{\tnr{j}}{\tnr{k}} \right)
 + \Phi_{im}\left( \dtree{\tnr{i}}{\tnr{j}}\right)\Phi(\bullet_k)
 & \int \int W_k \mltl dW_j \mltl dW_i  \\[3ex]
 \raisebox{-2.5ex}{ \dtree{\tnr{i}}{\tnl{j} \tnr{k} \tnr{l}}} &
 \Phi_{im}(\bullet_i) \Phi(\bullet_j)\Phi(\bullet_k)\Phi(\bullet_l)
 + \Phi_{im}\left(\dtree{\tnr{i}}{\tnr{j}}\right)\Phi(\bullet_k)\Phi(\bullet_l)
 + \Phi_{im}\left(\dtree{\tnr{i}}{\tnr{k}}\right)\Phi(\bullet_j)\Phi(\bullet_l)
 + \Phi_{im}\left(\dtree{\tnr{i}}{\tnr{l}}\right)\Phi(\bullet_j)\Phi(\bullet_k) & \raisebox{-2ex}{$\int W_j W_k W_l \mltl dW_i$}\\ &
  + \Phi_{im}\left(\dtree{\tnr{i}}{\tnl{j}\tnr{k}}\right)\Phi(\bullet_l)
  + \Phi_{im}\left(\dtree{\tnr{i}}{\tnl{j}\tnr{l}}\right)\Phi(\bullet_k)
  + \Phi_{im}\left(\dtree{\tnr{i}}{\tnl{k}\tnr{l}}\right)\Phi(\bullet_j)
  & \\[3ex]
  \raisebox{-1ex}{\dtree{\tnr{i}}{\tnl{j} \dtree{\tnr{k}}{\tnr{l}}}} &
  \Phi_{im}(\bullet_i)\Phi(\bullet_j)\Phi \left( \dtree{\tnr{k}}{\tnr{l}} \right)
  + \Phi_{im}\left(\dtree{\tnr{i}}{\tnr{j}}\right) \Phi\left( \dtree{\tnr{k}}{\tnr{l}} \right)
 + \Phi_{im}\left(\dtree{\tnr{i}}{\tnr{k}}\right)\Phi(\bullet_j)\Phi(\bullet_l)
+ \Phi_{im}\left(\dtree{\tnr{i}}{\tnl{j}\tnr{k}}\right) \Phi(\bullet_l)
 & \int \left(W_j \int W_l \mltl dW_k\right) \mltl dW_i \\[2ex]
 \raisebox{-1ex}{\dtree{\tnr{i}}{\dtree{\tnr{j}}{\tnl{k} \tnr{l}}}} &
 \Phi_{im}(\bullet_i)\Phi\left(\dtree{\tnr{j}}{\tnl{k} \tnr{l}} \right)
+ \Phi_{im} \left( \dtree{\tnr{i}}{\tnr{j}}\right)\Phi(\bullet_k)\Phi(\bullet_l)
+ \Phi_{im}\left(\raisebox{-1.5ex}{\dtree{\tnr{i}}{\dtree{\tnr{j}}{\tnr{k}}}} \right) \Phi(\bullet_l)
+ \Phi_{im}\left(\raisebox{-1.5ex}{\dtree{\tnr{i}}{\dtree{\tnr{j}}{\tnr{l}}}} \right) \Phi(\bullet_k) &
\int \int W_k W_l \mltl dW_j \mltl dW_i \\[3ex]
\raisebox{-3ex}{\dtree{\tnr{i}}{\dtree{\tnr{j}}{\dtree{\tnr{k}}{\tnr{l}}}}} &
\Phi_{im}(\bullet_i)\Phi\left(\raisebox{-1.5ex}{\dtree{\tnr{j}}{\dtree{\tnr{k}}{\tnr{l}}} }\right)
+ \Phi_{im} \left(\dtree{\tnr{i}}{\tnr{j}} \right) \Phi \left( \dtree{\tnr{k}}{\tnr{l}}\right)
+ \Phi_{im} \left( \raisebox{-1.5ex}{\dtree{\tnr{i}}{\dtree{\tnr{j}}{\tnr{k}}}} \right) \Phi(\bullet_{l}) &
\int\int\int W_l \mltl dW_k \mltl dW_j \mltl dW_i
 \end{array}
 \]
\label{default}
\end{sidewaystable}

To decide the weak order we will also need the B--series of the function
$f$, evaluated at the exact and the numerical solution.
From \thref{thm:Bex,thm:Bnum,lem:f_y} we obtain
\[
f(X(t_0+h))=\sum_{u\in U_f}\beta(u)\cdot\psi_\varphi(u)(h)\cdot G(u)(x_0),
\]
\[
f(Y_1)=\sum_{u\in U_f}\beta(u)\cdot\psi_\Phi(u)(h)\cdot G(u)(x_0),
\]
with \[\psi_\varphi([\emptyset]_f)\equiv1,\quad
\psi_\varphi(u=[\tau_1,\dotsc,\tau_\kappa]_f)(h) =
\prod\limits_{j=1}^\kappa \varphi(\tau_j)(h)
\]
and \[\psi_\Phi([\emptyset]_f)\equiv1,\quad
\psi_\Phi(u=[\tau_1,\dotsc,\tau_\kappa]_f)(h) =
\prod\limits_{j=1}^\kappa \Phi(\tau_j)(h).\]
With all the B--series in place, we can now present the order
conditions for the weak and strong convergence, for both the It\^{o}
and the Stratonovich case.
As usual we assume that method \eqref{eq:STM} is constructed such that $E\psi_{\Phi}(u)(h)=\mathcal{O}(h^{\rho(u)})$ $\forall u\in U_f$ and $\Phi(\tau)(h)=\mathcal{O}(h^{\rho(\tau)})$ $\forall\tau\in T$, respectively, where especially in the latter expression the $\mathcal{O}(\cdot)$-notation refers to the $L^2(\Omega)$-norm and $h\to0$. These conditions are fulfilled if for $\tau\in T$ and $k\in\N=\{0,1,\dots\}$ it holds that
$(\Phi_{ex}(\tau))^{2^k}=\mathcal{O}(h^{2^k\rho(\tau)})$ and
$(\Phi_{im}(\tau))^{2^k}=\mathcal{O}(h^{2^k\rho(\tau)})$.
Then we have weak consistency of order $p$ if and only if
\begin{equation}
    \label{eq:T_eq}
    \E\psi_\Phi(u)(h)=\E\psi_\varphi(u)(h)+\bO(h^{p+1})\quad\forall u\in U_f\text{ with }\rho(u)\leq p+\frac12
\end{equation}
(\eqref{eq:T_eq} slightly weakens conditions given in \cite{roessler06rta}), and
mean square global order $p$ if \cite{burrage04isr}
\begin{eqnarray}\label{eq:StrOrdCond1}
\Phi(\tau)(h)&=&\varphi(\tau)(h)+\bO(h^{p+\frac12})\quad\forall\tau\in T\text{ with }\rho(\tau)\leq p,\\
    \label{eq:StrOrdCond2}
\E\Phi(\tau)(h)&=&\E\varphi(\tau)(h)+\bO(h^{p+1})\quad\forall\tau\in T\text{ with }\rho(\tau)\leq p+\frac12
\end{eqnarray}
and all elementary differentials $F(\tau)$ fulfill a linear growth condition. Instead of the last requirement it is also enough to claim that there exists a constant $C$ such that $\|g_j'(y)\|\leq C\quad\forall y\in\real^m$, $j=0,\dots,M$, and all necessary partial derivatives exist \cite{burrage00oco}. 

%% file: numerics.tex
\newcommand{\alp}[1]{c_{#1}}
\newcommand{\I}[1]{I_{(#1)}}
\newcommand{\g}{g_{1}}
\newcommand{\f}{g_{0}}
\section{Numerical example}\label{seq:num}
As an example, we want to construct a parameterized family of strong order 1.5 Taylor methods, applicable to It\^{o}-SDEs with one-dimensional noise. As ansatz, we choose
\begin{align*}
Y_{n+1}=&B(\Phi_{ex},Y_n;h)+\alp1\I1g_{1,n+1}+\alp2hg_{0,n+1}+(\alp3\I{1,1}+\alp4h)g_{1,n+1}'g_{1,n+1}\\
&+\alp5\frac{h^2}2g_{0,n+1}'g_{0,n+1}+\alp6\frac{h^2}4g_{0,n+1}''(g_{1,n+1},g_{1,n+1}),
\end{align*}
where $g_{l,n+1}=g_l(Y_{n+1})$, $\alp1,\alp2,\alp3,\alp4,\alp5,\alp6\in\real$, $\Phi_{ex}(\tau)\equiv0$ for $\rho(\tau)>2$, and $\Phi_{ex}(\tau)$ is required to be deterministic for $\rho(\tau)=2$.

Together with the order conditions \eqref{eq:StrOrdCond1} and \eqref{eq:StrOrdCond2}, these requirements are sufficient to uniquely determine the remaining unknowns $\Phi_{ex}(\tau)$ for $\rho(\tau)\leq2$, and we obtain finally the following family of methods, using the abbreviation $g_{l}=g_l(Y_{n})$:
\begin{align*}
Y_{n+1}=&Y_n+\alp1\I1g_{1,n+1}+\alp2hg_{0,n+1}+(\alp3\I{1,1}+\alp4h)g_{1,n+1}'g_{1,n+1}
+\alp5\frac{h^2}2g_{0,n+1}'g_{0,n+1}
\\&+\alp6\frac{h^2}4g_{0,n+1}''(g_{1,n+1},g_{1,n+1})
+(1-\alp1)\I1\g+h(1-\alp2)\f
\\&+\left((-\alp1-\alp4)h+(1-2\alp1-\alp3)\I{1,1}\right)\g'\g
\\&+\left((1-\alp1)\I{0,1}-\alp1\I{1,0}\right)\g'\f
+(-\alp2\I{0,1}+(1-\alp2)\I{1,0})\f'\g
\\&+\left(\frac12\I{0,1}-\left(\frac32\alp1+\alp3+\alp4\right)h\I1
+(1-3\alp1-3\alp3)\I{1,1,1}\right)\g''(\g,\g)
\\&-\left((\alp1+\alp3+\alp4)h\I1+(3\alp1+3\alp3-1)\I{1,1,1}\right)\g'\g'\g
+(1-2\alp2-\alp5)\frac{h^2}2\f'\f
\\&-(\alp1+\alp4)h^2\g''(\f,\g)-\alp1\frac{h^2}2\g'\f'\g
-(\frac12\alp1+\alp4)h^2\g'\g'\f
\\&+\frac14(1-2\alp2-\alp6)h^2\f''(\g,\g)
-\alp3\frac{h^2}2\g'\g'\g'\g
-\frac12(\frac12\alp1+\alp3+\alp4)h^2\g'\g''(\g,\g)
\\&-(\alp1+\frac32\alp3+\alp4)h^2\g''(\g'g,\g)
-\frac12(\alp1+\alp3+\alp4)h^2\g'''(\g,\g,\g).
\end{align*}
In the following, we analyze numerically the order of convergence of this stochastic Taylor methods for the special parameter choice $\alp1=\alp2=\alp3=\alp4=\alp5=\alp6=\frac12$ by applying it to two examples. In both cases, the solution is approximated with step sizes $2^{-9}, \ldots, 2^{-15}$ and the sample average of $M=2000$ independent simulated realizations of the absolute error is calculated in order to estimate the expectation.

Our first example is the non-linear SDE \cite{kloeden99nso}
\begin{equation} \label{eq:nonlinex}
    X(t) = \int_0^t\left( \tfrac{1}{2} X(s) + \sqrt{X(s)^2 + 1} \right) \,
    ds + \int_0^t\sqrt{X(s)^2 + 1} \, dW(s)
\end{equation}
on the time interval $I=[0,1]$ with the solution $X(t) = \sinh (t + W(t))$.

As second example, we apply our Taylor method to the non-linear vector-SDE
\begin{equation}\label{eq:nonlinex2d}
\begin{pmatrix}
X_1(t)\\X_2(t)
\end{pmatrix}= \int_0^t\begin{pmatrix}
    \frac12X_1(s)+\sqrt{X_1(s)^2+X_2(s)^2+1}\\
    \frac12X_1(s)+\sqrt{X_2(s)^2+1}
    \end{pmatrix}\,
    dt +
    \int_0^t \begin{pmatrix}
    \cos(X_1(s))\\
    \sin(X_2(s))\\
    \end{pmatrix}\, dW(s),
\end{equation}
again on the time interval $I=[0,1]$. As here we do not know the exact solution, to approximate it we use the semi-implicit order 1.5 method due to Kloeden and Platen \cite{kloeden99nso} with a step size ten times smaller than the actual step size.

For both examples, the results at time $t=1$ are presented in Figure \ref{fig:3}, where the order of convergence corresponds to the slope of the regression line. As expected, we observe also numerically strong order 1.5.

\newlength{\figwidth}
\setlength{\figwidth}{0.485 \textwidth}
\begin{figure}[tbp]
\subfigure[Error when applied to  \eqref{eq:nonlinex}]{\includegraphics[width=\figwidth]{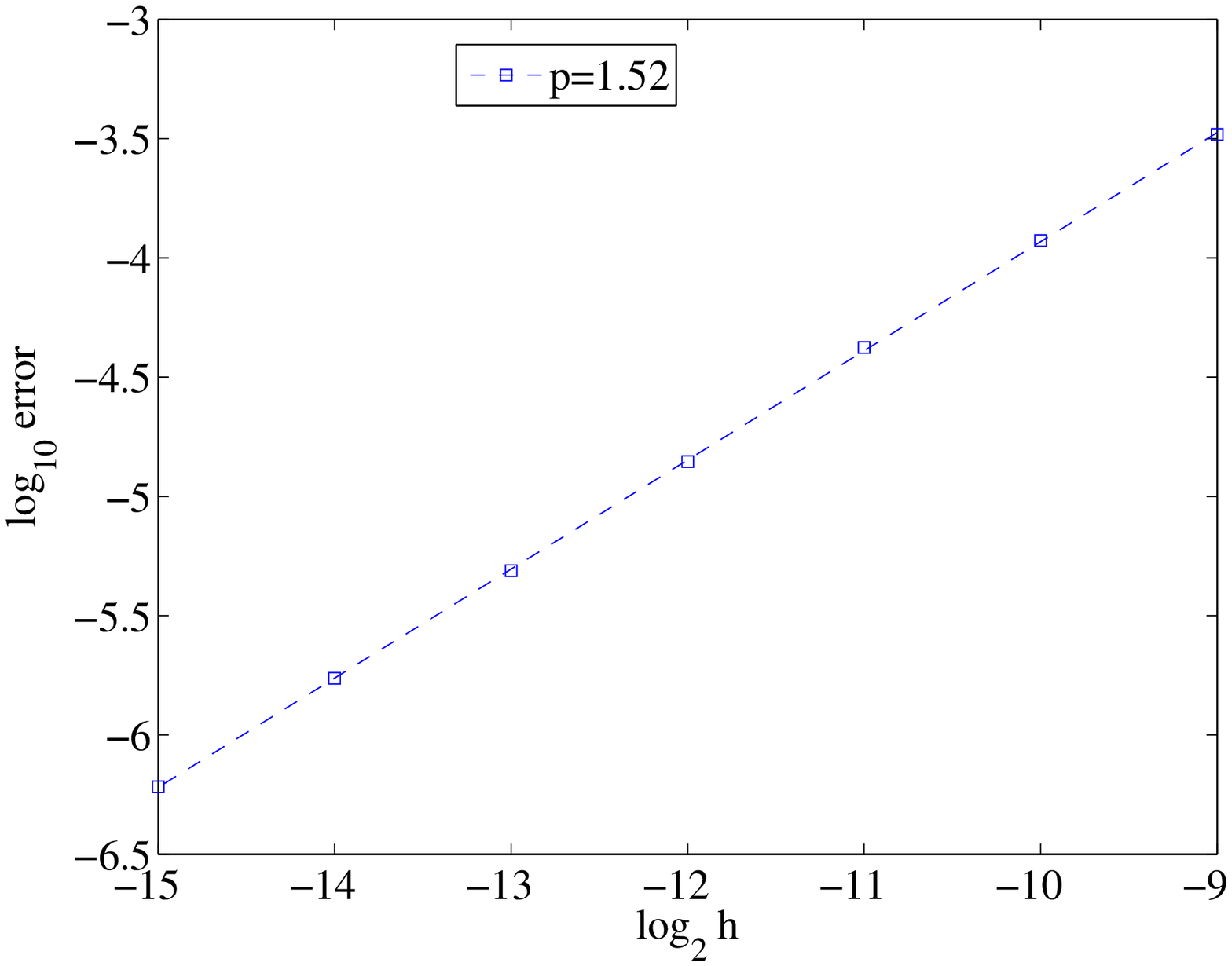}}
\subfigure[Error when applied to  \eqref{eq:nonlinex2d}]{\includegraphics[width=\figwidth]{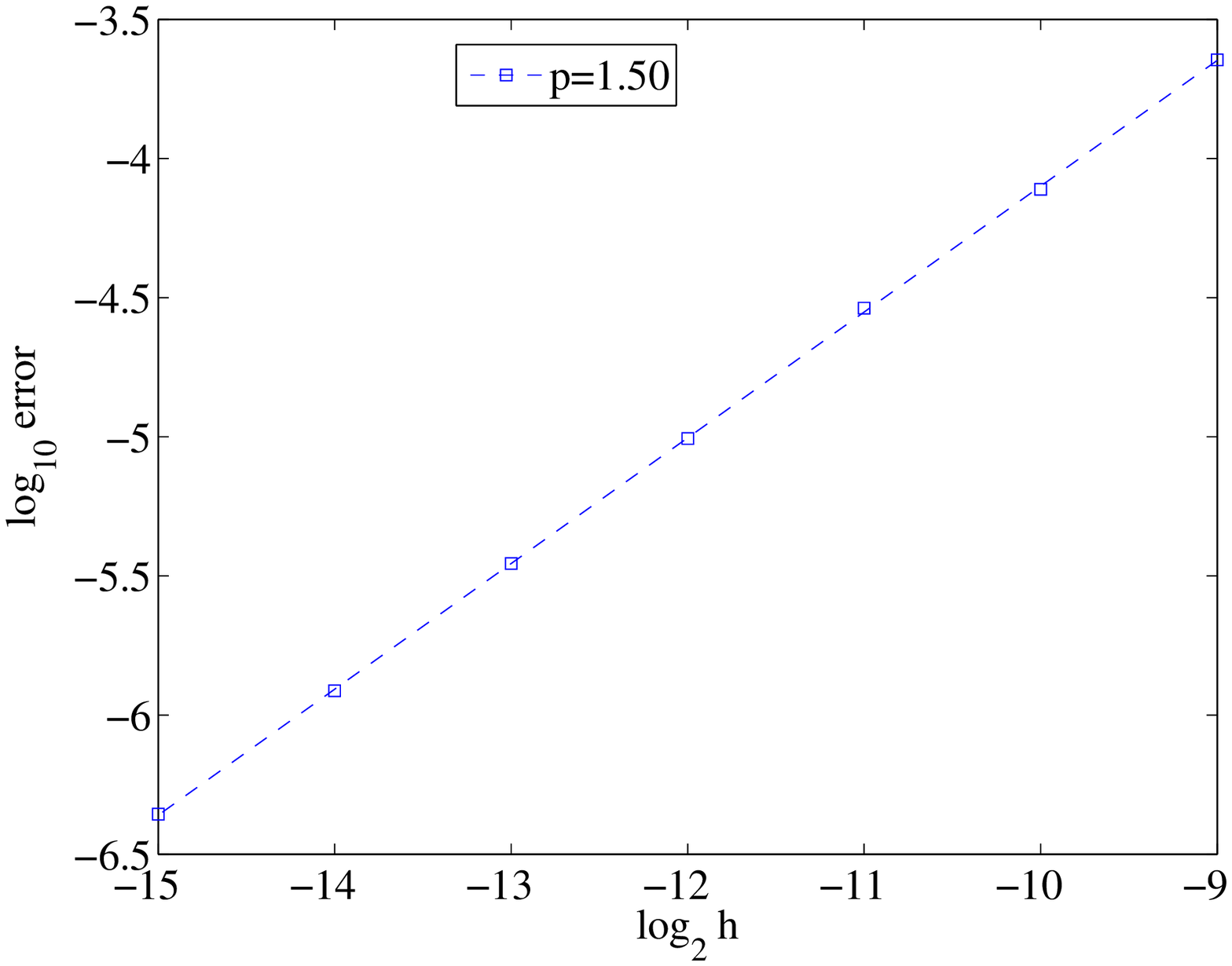}}
\caption{\label{fig:3}Error of example Taylor method vs.\ step size}
\end{figure}
\section{Conclusion}
In this article we have now given a new representation formula for (stochastic)
B--series evaluated in a B--series, which in contrast to the one known for the
deterministic case is not based on ordered trees. We have used this formula
for giving for the first time the order conditions of Taylor methods in
terms of rooted trees. We expect these to be useful for the construction of
new methods e.\,g.\ with good stability properties. Finally, we applied
these order conditions to derive in a simple manner a family of order 1.5
Taylor methods applicable to It\^{o} SDEs. In a forthcoming paper, the theory developed in this article is applied to analyze quite generally the order of iterated Taylor methods \cite{debrabantXXbao}. 